\documentclass{article}

\usepackage{amsmath}
\usepackage{amsfonts}
\usepackage{amssymb}
\usepackage{amsthm}

\title{A Stability Theorem for Matchings in Tripartite $3$-Graphs}
\author{Penny Haxell\thanks{Partially supported by NSERC. This author also thanks the Mittag-Leffler Institute in Djursholm, Sweden, where part of this work was done.}\; and Lothar Narins\\
Combinatorics and Optimization Department\\
Waterloo\\
Waterloo ON\\
Canada N2L 3G1}

\newtheorem{thm}{Theorem}[section]
\newtheorem{cor}[thm]{Corollary}
\newtheorem{lem}[thm]{Lemma}

\newtheorem{conj}{Conjecture}
\newtheorem*{claim}{Claim}
\newtheorem*{thm*}{Theorem}

\theoremstyle{definition}
\newtheorem{defn}[thm]{Definition}

\numberwithin{equation}{section}

\DeclareMathOperator{\lk}{lk}

\newcommand{\cF}{\mathcal{F}}
\newcommand{\cG}{\mathcal{G}}
\newcommand{\cH}{\mathcal{H}}

\newcommand{\cP}{\mathcal{P}}
\newcommand{\lset}[1]{\left\{#1\right\}}
\newcommand{\set}[2]{\left\{#1 : #2 \right\}}
\newcommand{\abs}[1]{\left|#1\right|}

\newcommand{\explode}{\divideontimes}

\begin{document}
\maketitle

\begin{abstract}
It follows from known results that every regular tripartite hypergraph of positive degree, with $n$ vertices in each class, has matching number at least $n/2$. This bound is best possible, and the extremal configuration is unique. Here we prove a stability version of this statement, establishing that every regular tripartite hypergraph with matching number at most $(1 + \varepsilon)n/2$ is close in structure to the extremal configuration, where ``closeness'' is measured by an explicit function of $\varepsilon$. We also answer a question of Aharoni, Kotlar and Ziv about matchings in hypergraphs with a more general degree condition.
\end{abstract}

\section{Introduction}

One of the simplest statements about matchings in bipartite graphs is the following corollary of Hall's Theorem.

\begin{thm} \label{cor:bipartitepm}
Let $G$ be a bipartite regular multigraph of positive degree. Then $G$ has a perfect matching. 
\end{thm}

Our principal aim in this paper is to study the hypergraph analogue of this result. A $k$-uniform multihypergraph (in which multiple edges are allowed), which we will call a $k$-graph for short, is \emph{$k$-partite} if its vertices can be partitioned into $k$ classes $V_1, \dots, V_k$ such that every edge has exactly one vertex from each class $V_i$. 

In this paper, we will limit our interests to $3$-partite $3$-graphs. For these, we have the following version of Theorem~\ref{cor:bipartitepm}.

\begin{thm} \label{thm:tripartitenhalf}
Let $\cH$ be a regular $3$-partite $3$-graph of positive degree, with $n$ vertices in each class. Then $\cH$ has a matching of size at least $\frac{n}{2}$. 
\end{thm}

This is an immediate consequence of a theorem of Aharoni~\cite{aharoni}, which verified the $3$-partite case of a famous old conjecture due to Ryser~\cite{ryser} relating the minimum size $\tau(\cH)$ of a vertex cover of $\cH$ (a set of vertices meeting all edges) to the maximum size $\nu(\cH)$ of a matching in $\cH$. 

\begin{thm}[Aharoni's Theorem] \label{thm:aharoni}
Let $\cH$ be a $3$-partite $3$-graph. Then $\tau(\cH) \leq 2\nu(\cH)$.
\end{thm}

\begin{proof}[Proof of Theorem~\ref{thm:tripartitenhalf}]
Let $\cH$ be an $r$-regular $3$-partite $3$-graph with $n$ vertices in each class. Then $\cH$ has $rn$ edges, but each vertex only intersects $r$ of them, hence any vertex cover must have at least $\frac{rn}{r} = n$ vertices, so $\tau(\cH) \geq n$. By Aharoni's Theorem, we have $\nu(\cH) \geq \frac{\tau(\cH)}{2} \geq \frac{n}{2}$, which proves the theorem.
\end{proof}

Theorem~\ref{thm:tripartitenhalf} is best possible, as can be seen by the following example. The \emph{truncated Fano Plane} $\cF$ (also called the \emph{Pasch configuration}) is the $3$-partite $3$-graph with six vertices $x_1$, $x_2$, $x_3$, $y_1$, $y_2$, $y_3$ and four edges $x_1x_2x_3$, $x_1y_2y_3$, $y_1x_2y_3$, $y_1y_2x_3$, where the sets $\lset{x_i, y_i}$ are the vertex classes. It is easy to check that $\cF$ is $2$-regular and $\nu(\cF) = 1$. For a hypergraph $\cH$ and an integer $s$, we denote by $s\cdot\cH$ the hypergraph with the same vertices as $\cH$ and with each edge replaced by $s$ parallel copies.

If $\cH$ consists of $\frac{n}{2}$ disjoint copies of $\frac{r}{2}\cdot\cF$, then $\nu(\cH) = \frac{n}{2}$, illustrating the tightness of Theorem~\ref{thm:tripartitenhalf} for every even $r$ and every even $n$. This is the unique extremal configuration, a fact which follows from~\cite{HNS2} in which the extremal hypergraphs for Aharoni's Theorem are characterized.

Our main aim in this paper is to prove the following stability version of Theorem~\ref{thm:tripartitenhalf}.

\begin{thm} \label{main}
Let $r \geq 2$. Let $\cH$ be an $r$-regular $3$-partite $3$-graph with $n$ vertices in each class, and let $\varepsilon \geq 0$. If $\nu(\cH) \leq (1 + \varepsilon)\frac{n}{2}$, then $\cH$ has at least $(1 - \left(22r - \frac{77}{3}\right)\varepsilon)\frac{n}{2}$ components that are copies of $\frac{r}{2}\cdot\cF$.
\end{thm}

In general one may expect stronger lower bounds on the matching number for \emph{simple} hypergraphs (i.e. those without multiple edges). For example Aharoni, Kotlar and Ziv~\cite{aharonikotlarziv} asked the following: when $r\geq 3$, does there exist $\mu = \mu(r) > 0$ such that $\nu(\cH) \geq (1 + \mu)\frac{\abs{A}}{2}$ for every simple $3$-partite $3$-graph $\cH$ with vertex classes $A$, $B$ and $C$ in which every vertex of $A$ has degree at least $r$ and every vertex of $B \cup C$ has degree at most $r$? The following weakened version of Theorem~\ref{main} answers this question affirmatively in a stronger form (with $\mu(r) = (72r^2 - 150r + 77)^{-1}$).

\begin{thm} \label{main2}
Let $r \geq 2$. Let $\cH$ be a $3$-partite $3$-graph with vertex classes $A$, $B$, and $C$, such that $\abs{A} = n$, and let $\varepsilon \geq 0$. Suppose that every vertex of $A$ has degree at least $r$, and that every vertex in $B \cup C$ has degree at most $r$. If $\nu(\cH) \leq (1 + \varepsilon)\frac{n}{2}$, then $\cH$ contains at least $(1 - (72r^2 - 150r + 77)\varepsilon)\frac{n}{2}$ disjoint copies of $\frac{r}{2}\cdot\cF$. 
\end{thm}

Theorem~\ref{main2} may be viewed as a direct hypergraph analogue of the corresponding weakening of Theorem~\ref{cor:bipartitepm}, with the condition that the minimum degree of vertices in vertex class $A$ is at least the maximum degree of vertices in class $B$, and which concludes that the bipartite graph has a matching of size $\abs{A}$. 

To prove Theorems~\ref{main} and~\ref{main2} we rely on a version of Hall's Theorem for hypergraphs, that uses a graph parameter $\eta$ whose definition is topological (the connectedness of the independence complex). However, the only properties of $\eta$ we will need come from known theorems which can be stated in purely graph theoretical terms. Thus none of our proofs will make any explicit reference to topology. This background material is described in Section~\ref{tool}. In Section~\ref{connbip} we prove a new lower bound on $\eta$ for line graphs of bipartite multigraphs, which will form the basis of our work in this paper. Section~\ref{stab} contains the proofs of Theorems~\ref{main} and~\ref{main2}, and in Section~\ref{last} we describe some constructions that show a limit on the amount by which our theorems could be improved. We close by mentioning a few open problems.

\section{Tools} \label{tool}

We begin by describing the version of Hall's Theorem for $k$-partite $k$-graphs that we will need. In this setting, the analogue of the neighbourhood of a vertex subset $S$ (which in the bipartite graph case is just an independent set of vertices) is a $(k - 1)$-partite $(k - 1)$-graph called the \emph{link} of $S$.

\begin{defn}
Let $\cH$ be a $k$-partite $k$-graph with vertex classes $V_1, \dots, V_k$, and let $S \subseteq V_i$. The \emph{link} of $S$ is the $(k - 1)$-partite $(k - 1)$-graph $\lk S$ whose vertex classes are the sets $\lset{V_1, \dots, V_k} \setminus \lset{V_i}$, and whose edges are $\set{e - v}{v \in S, v \in e \in E(\cH)}$.
\end{defn}

The generalization of Hall's Theorem to $k$-partite $k$-graphs~\cite{aharonihaxell, aharoniberger} can be stated in terms of a number of parameters of the link hypergraphs, for instance their matching numbers, or, as in its original formulation~\cite{aharonihaxell}, their \emph{matching width} (the maximum among all matchings of the size of the smallest matching intersecting each of its edges). The formulation we use here is based on the parameter $\eta(J)$, which is defined to be the topological connectedness of the independence complex of the graph $J$ plus $2$ (we add $2$ in order to make $\eta$ additive under disjoint union, which makes practically every formula involving it simpler. See e.g.~\cite{aharonibergerziv} for a discussion of this parameter.) Our graphs $J$ will usually be subgraphs of the line graph $L(G)$ of a bipartite graph $G$. The relevant version of Hall's Theorem for hypergraphs is as follows. 

\begin{thm}(Hall's Theorem for Hypergraphs) \label{thm:hallsforhypergraphs}
Let $\cH$ be a $k$-partite $k$-graph with vertex classes $V_1, \dots, V_k$, and let $d \geq 0$. If $\eta(L(\lk S)) \geq \abs{S} - d$ for every subset $S \subseteq V_i$, then $\cH$ has a matching of size at least $\abs{V_i} - d$.
\end{thm}

The only properties of $\eta$ we will need for our purposes are contained in the next three statements (and in fact the third follows easily from the second).

The first lemma is derived from basic properties of connectedness that can be found in any textbook on topology.

\begin{lem} \label{lem:connadditivity}
\begin{enumerate}
\item If the graph $J$ has no vertices then $\eta(J) = 0$.
\item If the graph $J$ contains an isolated vertex, then $\eta(J) = \infty$.
\item If $J$ and $K$ are disjoint graphs, then
\[
	\eta(J \cup K) \geq \eta(J) + \eta(K).
\]
\end{enumerate}
\end{lem}

Note that the last part implies in particular that adding any nonempty component to a graph increases its connectedness by at least $1$.

The next statement is Meshulam's Theorem~\cite{meshulam}, which relates $\eta(J)$ to that of two subgraphs of $J$, obtained by deleting an edge, or by what we call ``exploding'' an edge. If $J$ is a graph and $e \in E(J)$ is an edge, then we denote the edge deletion of
$e$ by $J - e$. We denote the edge explosion of $e$ by $J \explode e$, which is the subgraph of $J$ that remains after deleting both endpoints of $e$ and all their neighbours. 

\begin{thm}[Meshulam's Theorem] \label{thm:meshulam}
If $J$ is a graph and $e \in E(J)$, then
\[
	\eta(J) \geq \min(\eta(J - e), \eta(J \explode e) + 1).
\]
\end{thm}

This result (in a different formulation) is proved in~\cite{meshulam}. For more on Meshulam's Theorem see e.g.~\cite{adamaszekbarmak}, and~\cite{narins}, Section 5.3.

Various lower bounds on $\eta(J)$ in terms of other graph parameters have been proven, see e.g.~\cite{aharonibergerziv, meshulam}. Of particular interest to us is the following bound for line graphs (which was used for example in~\cite{aharonihaxell} but also follows easily from Theorem~\ref{thm:meshulam}).

\begin{thm} \label{thm:matchconn}
If $G$ is a multigraph, then
\[
	\eta(L(G)) \geq \frac{\nu(G)}{2}.
\]
\end{thm}

In the next section, we will apply Meshulam's Theorem to obtain an alternate version of the above bound for bipartite graphs, which takes into account the maximum degree as well as the matching number.

\section{The Connectedness of Line Graphs of Bipartite Multigraphs} \label{connbip}

In order to state and prove our results, we will need some definitions first.

If $G$ is a multigraph, and $J \subseteq L(G)$ is a subgraph of the line graph of $G$, we denote by $G_J$ the subgraph of $G$ with $V(G_J) = V(G)$ and $E(G_J) = V(J)$. Note that this makes sense, as the vertices of $J$ are a subset of the edges of $G$.

An \emph{$r$-regular $C_4$} is a bipartite multigraph consisting of a cycle of length $4$ and edges parallel to the edges of the cycle so that every vertex has degree $r$.

An edge $e \in E(J)$ is called \emph{decouplable} if $\eta(J - e) \leq \eta(J)$. It is called \emph{explodable} if $\eta(J \explode e) \leq \eta(J) - 1$. Note that by Meshulam's Theorem, every edge is either decouplable or explodable.

A graph is called \emph{reduced} if no edge is decouplable (hence every edge is explodable). A subgraph $J' \subseteq J$ is called a \emph{reduction} of $J$ if $J'$ is reduced, $V(J') = V(J)$, and $\eta(J') \leq \eta(J)$. Note that one may obtain a reduction of a graph $J$ by iteratively deleting decouplable edges until there are none left.

In the proof of our theorem, we will be applying Meshulam's Theorem to edges of the line graph, but will be regularly referring back to the original bipartite graph, whose edges are vertices of the line graph. To help eliminate confusion among vertices of the graph $G$, vertices of the line graph $L(G)$, edges of the graph, and edges of the line graph, we will use different terminology. Vertices and edges will always refer to vertices and edges of the original graph, while edges of the line graph will be called \emph{adjacencies}, or \emph{$J$-adjacencies} for $J$ a subgraph of the line graph. If a pair of edges of the graph intersect, they will be adjacent in the line graph, but not necessarily $J$-adjacent.

When talking about decouplable or explodable edges of the line graph, rather than say something like ``decouplable adjacency,'' we will often refer to these as decouplable (explodable) pairs of edges (of the original graph).

Our main aim in this section is to prove the following theorem.

\begin{thm} \label{thm:c4freeconnbound}
Let $G$ be a bipartite multigraph with maximum degree $r \geq 2$ that does not contain an $r$-regular $C_4$ component, and let $J \subseteq L(G)$. Then
\[
	\eta(J) \geq \frac{(2r - 3)\nu(G_J) + \abs{V(J)}}{6r - 7}.
\]
\end{thm}

Note that this is an improvement over the bound in Theorem~\ref{thm:matchconn} whenever $\abs{V(J)} \geq \frac{2r - 1}{2}\nu(G_J)$, and agrees with the bound when equality holds. In order to prove it, we will need the following lemma.

\begin{lem} \label{lem:explosiontypes}
Let $G$ be a bipartite multigraph with maximum degree $r \geq 2$ that does not contain an $r$-regular $C_4$ component, and let $J \subseteq L(G)$ be reduced and nonempty. Then if $\eta(J) \neq \infty$, $J$ contains an explodable pair $me$ of one of the following types:
\begin{enumerate}
\renewcommand{\theenumi}{(\arabic{enumi})}
\renewcommand{\labelenumi}{\theenumi}
\item $\nu(G_{J \explode me}) \geq \nu(G_J) - 1$ and $\abs{V(J \explode me)} \geq \abs{V(J)} - (3r - 2)$,
\item $\nu(G_{J \explode me}) \geq \nu(G_J) - 2$ and $\abs{V(J \explode me)} \geq \abs{V(J)} - (2r - 1)$, or
\item every reduction $J'$ of $J \explode me$ contains an explodable pair $m'e'$ such that $\nu(G_{J' \explode m'e'}) \geq \nu(G_J) - 3$, and $\abs{V(J' \explode m'e')} \geq \abs{V(J)} - (6r - 5)$.
\end{enumerate}
\end{lem}

\begin{proof}[Proof of Theorem~\ref{thm:c4freeconnbound} from Lemma~\ref{lem:explosiontypes}]
Let $G$ be a bipartite multigraph with maximum degree $r \geq 2$ that does not contain an $r$-regular $C_4$ component, and let $J \subseteq L(G)$. Also, suppose that $\abs{V(J)} \geq \frac{2r - 1}{2}\nu(G_J)$ (otherwise we may simply apply Theorem~\ref{thm:matchconn} to prove our theorem).

We construct a sequence of subgraphs $J_0, \dots, J_n$ with $J_0 = J$ and $J_n$ having no edges, in which $J_i$ is obtained from $J_{i - 1}$ by either deleting a decouplable $J_i$-adjacency or exploding an explodable pair of edges in $G_{J_i}$. This means that $\eta(J_{i - 1}) \geq \eta(J_i)$, with strict inequality whenever we perform an explosion.

We start by iteratively deleting decouplable adjacencies until we have a reduced subgraph $J_k \subseteq J$. Applying Lemma~\ref{lem:explosiontypes}, we find that there is an explodable pair of type (1), (2), or (3). We explode this pair to arrive at $J_{k + 1}$. In the case of an explosion of type (3), we then iteratively decouple decouplable pairs to arrive at a reduction $J'$ of $J_{k + 1}$ and then explode $m'e'$. We continue in this fashion until $J_n$ has no edges.

In the end, we will get a bound $\eta(J) \geq t + \eta(J_n)$, where $t$ is the number of explosions we perform in the sequence. Let $x_i$ denote the number of explosions of type ($i$). Note that for every explosion of type (3), we perform another explosion, so the total number of explosions is $t = x_1 + x_2 + 2x_3$. If $J_n$ has a vertex, it is isolated, which would show $\eta(J) = \infty$, so we may assume that $J_n$ is the empty graph, and so $\nu(G_{J_n}) = 0$ and $\eta(J_n) = 0$. Since the matching number is only affected by explosions, we thus obtain a bound
\[
	x_1 + 2x_2 + 3x_3 \geq \nu(G_J),
\]
since explosions of type ($i$) decrease the matching number by at most $i$. Similarly, these explosions must reduce the vertex number to $\abs{V(J_n)} = 0$, giving us the bound
\[
	(3r - 2)x_1 + (2r - 1)x_2 + (6r - 5)x_3 \geq \abs{V(J)}.
\]
Since we do not assume any control over the values of $x_i$, we suppose that we obtain the worst bound, where $t = x_1 + x_2 + 2x_3$ is minimized among all triples of non-negative integers $(x_1, x_2, x_3)$ satisfying the above two constraints. Relaxing the integer program to a linear program gives us the bound in the theorem, since for $\abs{V(J)} \geq \frac{2r - 1}{2}\nu(G_J)$, the minimum is obtained at
\[
	x_1 = 0, \quad x_2 = \frac{(6r - 5)\nu(G_J) - 3\abs{V(J)}}{6r - 7}, \quad x_3 = \frac{2\abs{V(J)} - (2r - 1)\nu(G_J)}{6r - 7},
\]
with a value of
\[
	t_{\mathrm{min}} = \frac{(2r - 3)\nu(G_J) + \abs{V(J)}}{6r - 7}.
\]

This can be confirmed by considering the dual linear program, which is to maximize $\nu(G_J)y_1 + \abs{V(J)}y_2$ among positive real pairs $(y_1, y_2)$ subject to the constraints
\begin{align*}
	y_1 + (3r - 2)y_2 \leq 1,\\
	2y_1 + (2r - 1)y_2 \leq 1,\\
	3y_1 + (6r - 5)y_2 \leq 2.
\end{align*}
It is enough to note that
\[
	y_1 = \frac{2r - 3}{6r - 7}, \quad y_2 = \frac{1}{6r - 7}
\]
is feasible for the dual program, and its value is $\nu(G_J)y_1 + \abs{V(J)}y_2 = t_{\mathrm{min}}$.
\end{proof}

\begin{proof}[Proof of Lemma~\ref{lem:explosiontypes}]
Let $G$ be a bipartite multigraph with maximum degree $r \geq 2$, and let $J \subseteq L(G)$ be reduced and contain an edge. Suppose that there are no explodable pairs of any of the types (1), (2), and (3). We aim to show that $G$ contains an $r$-regular $C_4$ component. We follow along the lines of~\cite{HNS}, using many of the same ideas and techniques.

Note that any explosion in $J$ destroys at most $3r - 2$ edges of $G$. Indeed, any pair of intersecting edges only have three vertices in which to meet other edges, and as $G$ has maximum degree $r$, there are only $3r - 2$ edges incident to those three vertices, because the two edges in question count towards the degree of two of these vertices each. Thus, every explosion that reduces the matching number by at most $1$ is automatically an explosion of type (1).

\begin{lem} \label{lem:parallel}
No two edges that are parallel are $J$-adjacent.
\end{lem}

\begin{proof}
If $e$ and $f$ are parallel, then $\nu(G_{J \explode ef}) \geq \nu(G_J) - 2$, and $\abs{V(J \explode ef)} \geq \abs{V(J)} - (2r - 2)$, so this would be an explosion of type (2), which does not exist. Hence $e$ and $f$ cannot be $J$-adjacent, as $J$ is reduced.
\end{proof}

\begin{lem} \label{lem:mdegree}
If $M \subseteq V(J)$ is a maximum matching of $G_J$, and $e \in V(J) \setminus M$ is $J$-adjacent to an edge of $M$, then $e$ is $J$-adjacent to two edges of $M$ (one at each endpoint of $e$).
\end{lem}

\begin{proof}
Suppose $e$ is $J$-adjacent to only $m \in M$, but no other edge of $M$. Then exploding $me$ would destroy only one edge of $M$, which reduces the matching number by at most $1$, hence this would be an explosion of type (1), which we assume not to exist. Thus, $e$ must be $J$-adjacent to a second edge of $M$.
\end{proof}

We now make a few definitions, which will provide the setup for the two upcoming Lemmas~\ref{lem:ysaturated} and~\ref{lem:edgesoutsidex}.

For a maximum matching $M \subseteq V(J)$ and two edges $m \in M$, and $e \in V(J) \setminus M$ with $me \in E(J)$, define $\cP(M, m, e)$ to be the set of edges in $V(J)$ contained in some $M$-alternating path in $G_J$ starting with $m$, $e$. Let $A$ be the vertex class of $G$ containing the starting point of these paths, and let $B$ be the other. Let $Y \subseteq A$ be the set of vertices in edges of $\cP(M, m, e)$ contained in $A$, but not including the vertices of $m$ and $e$. Let $X \subseteq B$ be the set of vertices in edges of $\cP(M, m, e)$ contained in $B$, this time including the vertex in $m \cap e$.

Let $m' \in M$ be the other edge of $M$ besides $m$ that is $J$-adjacent to $e$, which is guaranteed to exist by Lemma~\ref{lem:mdegree}.

\begin{lem} \label{lem:ysaturated}
All vertices of $Y$ are $M$-saturated.
\end{lem}

\begin{proof}
Suppose $y \in Y$ is $M$-unsaturated. By the definition of $Y$, there is an $M$-alternating path in $G_J$ starting $m$, $e$, and ending in vertex $y$. Exploding $me$ destroys two edges $m$ and $m'$ of $M$, since it is not of type (1). However, for $M' = M \setminus \lset{m, m'}$, we have that the rest of the path ending in $y$ is an $M'$-augmenting path in $G_{J \explode me}$, which means that in fact $\nu(G_{J \explode me}) \geq \nu(G_J) - 1$, and therefore the explosion of $me$ is of type (1) after all. This is a contradiction, thus no $y \in Y$ can be $M$-unsaturated.
\end{proof}

\begin{lem} \label{lem:edgesoutsidex}
Every edge of $M$ with a vertex in $Y$ is $J$-adjacent in $Y$ to an edge whose other endpoint is not in $X$.
\end{lem}

\begin{proof}
Consider what happens when we explode $me$. This destroys $m$ and $m'$. Let $d$ be the vertex of $G_J$ in $m' \cap X$. Let $J'$ be a reduction of $J \explode me$, and let $M' = M \setminus \lset{m, m'}$. We will make use of the fact that $me$ is not an explosion of type (3). This means that $J'$ does not contain a pair of $J'$-adjacent edges whose explosion would reduce the matching number by at most $1$ and destroy at most $3r - 3$ edges.

\begin{claim}
All edges of $M'$ with a vertex in $Y$ are not $J'$-adjacent to any edge preceding or succeeding them in an $M'$-alternating path in $G_{J'}$ starting at $d$.
\end{claim}

\begin{proof}
Consider any $M'$-alternating path $P$ in $G_{J'}$ starting at $d$. Since these are all parts of the $M$-alternating paths in $G_J$ starting with $m$, $e$, we see that every edge of $M'$ incident to $X$ is in one of these paths. Note that $d$ has degree at most $r - 1$ in $G_{J'}$, since $m'$ was incident to it and was destroyed in the explosion of $me$. Denote the edges of the path $P$ by $e_1, m_1, e_2, m_2, \dots$, so that $m_i \in M'$ and $e_1$ is incident to $d$. We claim that none of the pairs in the path are $J'$-adjacent. Indeed, $e_1$ and $m_1$ are not, because if they were explodable, this would make $me$ an explosion of type (3). To see this, note that since we only destroy one edge of $M'$ in the second explosion, we reduce $\nu(G_J')$ by at most $1$, and since $d$ has degree at most $r - 1$, we destroy at most $3r - 3$ edges in the second explosion. This kind of explosion has been ruled out. Neither are $m_1$ and $e_2$ $J'$-adjacent, since exploding this pair would not destroy $e_1$, which means we could add it to $M' \setminus \lset{m_1, m_2}$ to have a matching of size $\nu(G_J') - 1$ after the second explosion, and again we destroy at most $3r - 3$ edges incident to $e_1 \cap m_1$, since we don't destroy $e_1$. This would again make $me$ an explosion of type (3), which contradicts our assumptions.

Continuing in this fashion along the path, we see that $e_i$ and $m_i$ are not $J'$-adjacent, because exploding this pair would reduce the matching number by at most $1$, as $e_i$ is not $J'$-adjacent to $m_{i - 1}$, and for the same reason, we only destroy $3r - 3$ edges in the second explosion, which would make $me$ an explosion of type (3). Next, we see that $m_i$ and $e_{i + 1}$ are not $J'$-adjacent, because exploding this pair would leave an $(M' \setminus \lset{m_i, m_{i + 1}})$-augmenting path $e_1, m_1, \dots, e_i$, so even though two edges of $M'$ are destroyed, the matching number decreases only by $1$, if at all, and again, we only destroy $3r - 3$ edges in this second explosion because $e_i$ is not destroyed. This proves the claim.
\end{proof}

\begin{claim}
Every edge of $M'$ incident to $Y$ is not $J'$-adjacent to any edge between $X$ and $Y$.
\end{claim}

\begin{proof}
Consider any pair of intersecting edges $m'' \in M'$ and $e' \in V(J') \setminus M'$ that go between $X$ and $Y$. We claim that if these were explodable, then $me$ would be an explosion of type (3), and hence these are not $J'$-adjacent, as $J'$ is reduced.

If $e'$ is incident to $b \in m \cap e$, then exploding $m''e'$ reduces $\nu(G_{J'})$ by only $1$ and destroys at most $3r - 4$ edges, since $m$ and $e$ are already gone. This would make $m$ an explosion of type (3). If $e'$ is incident to $d$, then it is the predecessor of $m''$ on some $M'$-alternating path, so they are not $J'$-adjacent by the previous claim. Otherwise, $e'$ is incident to a vertex of $X \setminus \lset{b, d}$. If it is parallel to $m''$, then exploding it would destroy one edge of $M'$ and at most $2r - 2$ edges, which would again make $me$ a type (3) explosion.

The only remaining possibility is that $e'$ meets an edge $m''' \in M'$ in a vertex of $X$. If there is an $M'$-alternating path from $d$ to $m'''$ that does not use $m''$, appending $e'$ and $m''$ to this path shows by the previous claim that $e'$ and $m'$ are not $J'$-adjacent. If there is no such path, then $e'$ together with the part between $m''$ and $m'''$, inclusive, of an $M'$-alternating path from $d$ to $m'''$ forms an $M'$-alternating cycle. In this case, let $M''$ be obtained from $M'$ by switching on that $M'$-alternating cycle. Now exploding $m''e'$ only destroys one edge of $M''$, so the resulting graph has a matching of size at least $\nu(G_{J'}) - 1$. The explosion also does not destroy a predecessor of $m''$ on some $M'$-alternating path from $d$, so we lose at most $3r - 3$ edges in the second explosion, which makes $me$ of type (3).
\end{proof}

Thus every edge of $M'$ incident to $Y$ is not $J'$-adjacent to any edge between $X$ and $Y$. However, none of these edges are isolated in $J'$, since we have $\eta(J') \leq \eta(J) - 1 < \infty$. This means that they each must be $J'$-adjacent to some edge that is not between $X$ and $Y$. If this edge is incident to $X$, we would have an $M'$-augmenting path by going from $d$ to the matching edge then to this edge, so the edge is not incident to $X$, which proves Lemma~\ref{lem:edgesoutsidex}, since $J'$-adjacent implies $J$-adjacent.
\end{proof}

We now complete the proof of Lemma~\ref{lem:explosiontypes}.

Choose the triple $(M, m, e)$ consisting of a maximum matching $M$ of $G_J$ and a pair of $J$-adjacent edges $m \in M$ and $e \in V(J) \setminus M$ so that $\abs{\cP(M, m, e)}$ is maximized among all such triples. We claim that $m$ and $e$ are in fact part of an $r$-regular $C_4$ component of $G_J$. Let $m'$ be the other edge of $M$ that is $J$-adjacent to $e$, which exists by Lemma~\ref{lem:mdegree}, and let the vertices of $m$, $e$, and $m'$ be $a$, $b$, $c$, and $d$, with $m = ab$, $e = bc$, and $m' = cd$.

First, we show that there are no edges $J$-adjacent to $m$ at $a$ that do not go to $d$. Suppose that $e'$ were such an edge. By Lemma~\ref{lem:mdegree}, it is $J$-adjacent to another edge $\hat{m} \in M$. If $e' \cap \hat{m} \not\subseteq X$, then we have a contradiction, as any edge in $\cP(M, m, e)$ can be reached by an $M$-alternating path starting with $\hat{m}$, $e'$, then continuing with $m$, $e$, and the rest of the path that shows it is in $\cP(M, m, e)$. But $\hat{m} \notin \cP(M, m, e)$, since it is not incident to $X$, which runs contrary to the assumption that $\abs{\cP(M, m, e)}$ is maximum. Therefore, $\hat{m}$ must be incident to $X$. If $\hat{m} \neq m'$, then $\hat{m}$ is also incident to $Y$, and so by Lemma~\ref{lem:edgesoutsidex}, it has an edge $e''$ $J$-adjacent to it in $Y$, which is not incident to $X$, and by Lemma~\ref{lem:mdegree}, $e''$ is $J$-adjacent to another edge $\hat{m}' \in M$. But then $\cP(M, \hat{m}', e'')$ would strictly contain $\cP(M, m, e)$. This is because for any edge in $\cP(M, m, e)$, if the path from $m$, $e$ containing it passes through $\hat{m}$, we can start with $\hat{m}'$, $e''$, $\hat{m}$ and continue along the path to reach it from $\hat{m}'$, $e''$. If on the other hand the path from $m$, $e$ does not include $\hat{m}$, we can reach it by starting with $\hat{m}'$, $e''$, $\hat{m}$, $e'$, $m$, $e$, and continuing along the path. This also contradicts our choice of $(M, m, e)$. This means the only option is $\hat{m} = m'$.

Next, we establish that there is an edge $f = ad$, which is $J$-adjacent to $m$. If there were no such edge, then exploding $me$ would destroy only edges incident to $b$ and $c$, of which there are at most $2r - 1$, since $bc$ is an edge. Since also $\nu(G_J)$ would be reduced by at most $2$, this would be an explosion of type (2), which we assume not to exist. Thus there must be an edge incident to $a$ that is $J$-adjacent to $m$, and by the argument in the previous paragraph, we have seen that such an edge must be incident to $d$.

Now consider the matching $M^\times = M \cup \lset{e, f} \setminus \lset{m, m'}$, obtained by switching $M$ along the $C_4$ on $abcd$. Note that $\cP(M^\times, e, m) = \cP(M, m, e) \cup \lset{f} \setminus \lset{m'}$, since any $M$-alternating path starting $m$, $e$, $m'$ can be converted to an $M^\times$-alternating path by starting with $e$, $m$, $f$, and continuing the same way. Therefore this triple is also maximizing, so the same argument as above applies to show that the only edges $J$-adjacent to $e$ at $c$ are parallel to $m'$.

We now show that $e$ and $f$ have no $J$-neighbours at $a$ or $c$, respectively, except those parallel to $m$ and $m'$, respectively. If there were an edge $g$ contradicting this statement, then by switching to $M^\times$ and applying Lemma~\ref{lem:mdegree}, we would find that $g$ is $J$-adjacent to some other edge $h$ of $M^\times$ not among $\lset{e, f}$. But $h$ is also an edge of $M$, hence by Lemma~\ref{lem:mdegree}, it would need to be $J$-adjacent to a second edge of $M$, which by virtue of being incident to $a$ or $c$ would have to be $m$ or $m'$. But as seen above, no such edge is $J$-adjacent to $m$ or $m'$, thus we have a contradiction. This shows that none of $m$, $m'$, $e$, and $f$ have any $J$-neighbours incident to $\lset{a, c}$ that leave the $C_4$ on $abdc$.

Now suppose that there is an edge incident to $d$ that is not incident to $a$ or $c$. Such an edge is disjoint from $m$ and $e$, so it survives the explosion of $me$. By what we have proven above, the explosion of $me$ only destroys edges incident to $b$ and $d$, of which there are at most $2r$. But since at least one edge incident to $d$ survives, the explosion would destroy at most $2r - 1$ edges, and it clearly only destroys $2$ edges of $M$, hence this would be an explosion of type (2). Therefore, there are no edges incident to $d$, except those that go to $a$ or $c$. A similar argument, by threatening to explode $m'f$, shows that there are no edges incident to $b$, except those that go to $a$ or $c$. If any of $b$ or $d$ is not of degree $r$, then $me$ would again be an explosion of type (2), so they are both maximum degree vertices. This forces all edges incident to $a$ and $c$ to be those from $b$ and $d$ by a simple counting argument. Therefore, $abcd$ form the vertices of an $r$-regular $C_4$-component of $G_J$. This proves the lemma by contraposition.
\end{proof}

\begin{cor} \label{cor:connbound}
Let $G$ be a bipartite multigraph with maximum degree $r \geq 2$ that contains at most $k$ components that are $r$-regular $C_4$'s. Then
\[
	\eta(L(G)) \geq \frac{(2r - 3)\nu(G) + \abs{E(G)} - k}{6r - 7}.
\]
\end{cor}

\begin{proof}
Assume, without loss of generality, that $G$ has exactly $k$ components that are $r$-regular $C_4$'s. Let $G'$ be equal to $G$ with all its $r$-regular $C_4$ components removed. We have $\abs{E(G')} = \abs{E(G)} - 2rk$ and $\nu(G') = \nu(G) - 2k$. Applying Theorem~\ref{thm:c4freeconnbound} to $G'$, we have
\[
	\eta(L(G')) \geq \frac{(2r - 3)\nu(G') + \abs{E(G')}}{6r - 7}.
\]
Adding $k$ non-empty components to $L(G')$ will increase its connectedness by at least $k$ by Lemma~\ref{lem:connadditivity}, so $\eta(L(G)) \geq \eta(L(G')) + k$, and this gives the desired bound via a straightforward calculation.
\end{proof}

We remark that Theorem~\ref{thm:c4freeconnbound} is tight when $r = 2$, as can be seen by taking $G$ to be the disjoint union of any number of paths $P_4$ of length $3$ and cycles of length $10$ (since $\eta(P_4) = 1$, and $\eta(C_{10}) = 3$).

\section{Stability}\label{stab}

We have two versions of our stability theorem. One is for $r$-regular $3$-partite $3$-graphs, and the other has slightly less stringent degree conditions, which of course results in a weaker bound.

\begin{thm} \label{thm:regularfstability}
Let $r \geq 2$. Let $\cH$ be an $r$-regular $3$-partite $3$-graph with $n$ vertices in each class, and let $\varepsilon \geq 0$. If $\nu(\cH) \leq (1 + \varepsilon)\frac{n}{2}$, then $\cH$ has at least $(1 - \left(22r - \frac{77}{3}\right)\varepsilon)\frac{n}{2}$ components that are $\frac{r}{2}\cdot\cF$'s.
\end{thm}

\begin{thm} \label{thm:abcfstability}
Let $r \geq 2$. Let $\cH$ be a $3$-partite $3$-graph with vertex classes $A$, $B$, and $C$, such that $\abs{A} = n$, and let $\varepsilon \geq 0$. Suppose that every vertex of $A$ has degree at least $r$, and that every vertex in $B \cup C$ has degree at most $r$. If $\nu(\cH) \leq (1 + \varepsilon)\frac{n}{2}$, then $\cH$ contains at least $(1 - (72r^2 - 150r + 77)\varepsilon)\frac{n}{2}$ disjoint copies of $\frac{r}{2}\cdot\cF$.
\end{thm}

Our strategy is to use the low matching number to find a subset of each vertex class whose links have low connectedness. From this, we deduce that each link must have many $r$-regular $C_4$ components. We analyze how these can interact and deduce that a number of them must extend to $\frac{r}{2}\cdot\cF$'s. We break the proofs down into several lemmas that apply in both situations.

\begin{lem} \label{lem:c4links}
Let $\cH$ be a $3$-partite $3$-graph with vertex classes $A$, $B$, and $C$, such that $\abs{A} = n$, and let $\varepsilon \geq 0$. Suppose that every vertex of $A$ has degree at least $r$, and that every vertex in $B \cup C$ has degree at most $r$. If $\nu(\cH) \leq (1 + \varepsilon)\frac{n}{2}$, then $\lk A$ contains at least $(1 - (6r - 7)\varepsilon)\frac{n}{2}$ components that are $r$-regular $C_4$'s.
\end{lem}

\begin{proof}
We know that there must be some $S \subseteq A$ such that $\eta(L(\lk S)) \leq \abs{S} - (n - \nu(\cH))$, otherwise $\cH$ would have a matching larger than $\nu(\cH)$ by Theorem~\ref{thm:hallsforhypergraphs}. Now $\lk S$ has at least $r\abs{S}$ edges and maximum degree at most $r$, so $\tau(\lk S) \geq \abs{S}$, and so by K\"onig's Theorem it follows from this that $\nu(\lk S) \geq \abs{S}$.

Let $k$ be the number of $r$-regular $C_4$ components of $\lk S$. By Corollary~\ref{cor:connbound}, we have
\begin{align*}
	\eta(L(\lk S)) &\geq \frac{(2r - 3)\nu(\lk S) + \abs{E(\lk S)} - k}{6r - 7}\\
	&\geq \frac{(2r - 3)\abs{S} + r\abs{S} - k}{6r - 7}\\
	&= \frac{(3r - 3)\abs{S} - k}{6r - 7}.
\end{align*}
Combining this with our upper bound, we find
\begin{align*}
	k &\geq (6r - 7)(n - \nu(\cH)) - (3r - 4)\abs{S}\\
	&\geq (6r - 7)\left(n - (1 + \varepsilon)\frac{n}{2}\right) - (3r - 4)n\\
	&= (1 - (6r - 7)\varepsilon)\frac{n}{2}.
\end{align*}
Since the vertices of an $r$-regular $C_4$ have degree $r$, which is the maximum degree of any vertex in $B \cup C$, no additional edges of $\lk A$ intersect any of these components of $\lk S$, hence these are indeed components of $\lk A$, which proves our lemma.
\end{proof}

We say a subgraph of a link of $\cH$ \emph{hosts} an edge $e$ of $\cH$ if the edge of the link corresponding to $e$ is present in the subgraph.

\begin{lem} \label{lem:extendtof}
Let $\cH$ be a $3$-partite $3$-graph, let $A$ be one of its vertex classes, and suppose that every vertex in $A$ has degree at most $r$. If an $r$-regular $C_4$ in $\lk A$ does not host two disjoint edges of $\cH$, then the edges it hosts form a copy of $\frac{r}{2}\cdot\cF$.
\end{lem}

\begin{proof}
Let $e$, $f$, $g$, and $h$ be pairwise nonparallel edges of the $r$-regular $C_4$ in $\lk A$, so that $e, f$ and $g, h$ form matchings. Since no pair of edges extend to disjoint edges of $\cH$, all $e$-parallel and $f$-parallel edges must meet in the same vertex, and similarly, all $g$-parallel and $h$-parallel edges meet in the same vertex. These, however, must be two different vertices, since they are incident to $2r$ edges altogether. Thus, each of these vertices is incident to $r$ edges, and so there are $r$ total $e$-parallel and $f$-parallel edges, and $r$ total $g$-parallel and $h$-parallel edges. To form an $r$-regular $C_4$, there must be the same number of $e$-parallel edges as $f$-parallel ones, and similarly the same number of $g$-parallel and $h$-parallel edges. Thus there must be $\frac{r}{2}$ of each, and this forms an $\frac{r}{2}\cdot\cF$, as desired.
\end{proof}

\begin{lem} \label{lem:perfectmatchingcomp}
Let $\cH$ be a $3$-partite $3$-graph. If an $r$-regular $C_4$ component $K$ of a link of a vertex class of $\cH$ is host to two disjoint edges of $\cH$, and all of the vertices of $K$ are part of $r$-regular $C_4$ components of the links of the other vertex classes, then $K$ belongs to a component of $\cH$ that either
\begin{enumerate}
\renewcommand{\theenumi}{(\arabic{enumi})}
\renewcommand{\labelenumi}{\theenumi}
\item has $2$ vertices in each class and a matching of size $2$, or
\item has $4$ vertices in each class and a matching of size $4$.
\end{enumerate}
In particular, $K$ belongs to a component of $\cH$ with a perfect matching.
\end{lem}

\begin{proof}
Let $V_1$, $V_2$, and $V_3$ be the vertex classes of $\cH$, and suppose that the $r$-regular $C_4$ component $K$ in question is a component of $\lk V_1$.

Let $a_1 a_2 a_3$ and $b_1 b_2 b_3$ be two disjoint edges of $\cH$ with $a_i, b_i \in V_i$ and $a_2$, $b_2$, $a_3$, and $b_3$ being the vertices of an $r$-regular $C_4$ component of $\lk V_1$, all of whose vertices are part of $r$-regular $C_4$ components in the other links. We consider two cases:

\noindent \emph{Case 1}. $a_1 a_2$ and $b_1 b_2$ belong to the same $r$-regular $C_4$ component of $\lk V_3$.

In this case, all edges incident to $a_1$ or $b_1$ are incident to $a_2$ or $b_2$, hence incident to $a_3$ or $b_3$, and vice versa. Thus the $a_i$ and $b_i$ are the vertices of a component of type (1).

\noindent \emph{Case 2}. $a_1 a_2$ and $b_1 b_2$ belong to two different $r$-regular $C_4$ components of $\lk V_3$.

In this case, let the vertices of the components be $a_1$, $c_1$, $a_2$, $c_2$, and $b_1$, $d_1$, $b_2$, $d_2$, respectively. Now consider $\lk V_2$. It has edges $a_1 a_3$ and $b_1 b_3$. If $a_1 b_3$ were an edge of $\lk V_2$, then $a_1$, $b_1$, $a_3$, and $b_3$ would be the vertices of an $r$-regular $C_4$ component in $\lk V_2$, which would preclude the existence of any edge between $a_3$ or $b_3$ and $c_1$. But any edge of $\cH$ corresponding to $c_1 a_2$ in $\lk V_3$ must be incident to $a_3$ or $b_3$ as seen by looking at $\lk V_1$. This contradiction implies that $a_1 a_3$ and $b_1 b_3$ are in separate components of $\lk V_2$, and thus the edges of $\cH$ corresponding to $a_2 b_3$ in $\lk V_1$ must extend to $c_1$, rather than $a_1$ (these being the only two options given by $\lk V_3$). A similar argument shows that edges corresponding to $b_2 a_3$ extend to $d_1$. Now by assumption, $a_3$ and $b_3$ are each part of an $r$-regular $C_4$ component of $\lk V_2$, and given the edges we already have shown to exist, we know that these are two distinct components, and we know three vertices of each. Denote the remaining vertices by $d_3$ and $c_3$, respectively, so that $a_1$, $d_1$, $a_3$, $d_3$ are the vertices of one component, and $b_1$, $c_1$, $b_3$, $c_3$ the vertices of the other component.

Since $a_3$ and $c_1$ are in distinct components of $\lk V_2$, we see that all edges of $\cH$ corresponding to $a_2 a_3$ extend to $a_1$. Similarly, all edges corresponding to $b_2 b_3$ extend to $b_1$, all the ones corresponding to $a_2 b_3$ extend to $c_1$, and $b_2 a_3$ to $d_1$. Now in $\lk V_2$ there are the edges $a_1 d_3$ and $b_1 c_3$. These do not extend to $a_2$ or $b_2$ as seen in $\lk V_1$, and hence must extend to $c_2$ and $d_2$, respectively, by considering $\lk V_3$. Similarly, the edges $c_1 c_3$ and $d_1 d_3$ in $\lk V_2$ must extend to $c_2$ and $d_2$, respectively.

Thus, we have deduced the structure of the subgraph $\cG$ of $\cH$ induced by these twelve vertices. It has $4$ vertices in each class and a matching $a_1 a_2 a_3$, $b_1 b_2 b_3$, $c_1 c_2 c_3$, $d_1 d_2 d_3$ of size $4$. All that remains to complete the proof is to show that this is a component of $\cH$, which would make it a component of type (2).

Suppose there were an edge $e$ of $\cH$ containing a vertex $u$ of $\cG$ and a vertex $v$ not in $\cG$. Let $V_i$ be the vertex class of $u$, let $V_j$ be the vertex class of $v$, and let $V_k$ be the third vertex class of $\cH$. The presence of $e$ would mean that there is an edge $uv$ in $\lk V_k$. But since the parts of $\cG$ present in the links $\lk V_2$ and $\lk V_3$ are components of those links, $uv$ cannot be part of these links, and hence $k = 1$. Now consider the third vertex $w$ of $e$, which is in $V_1$. If $w$ is a vertex of $\cG$, then $vw$ is an edge of $\lk V_i$ of the type we just excluded, and if $w$ is a vertex not in $\cG$, then $uw$ is an edge of $\lk V_j$ giving us a similar contradiction. Thus no such edge $e$ can exist, and $\cG$ is indeed a component of $\cH$.

As these cases were exhaustive, the claim follows.
\end{proof}

We remark that with the previous three lemmas in hand, it would be a short step to conclude that any $3$-partite $3$-graph satisfying the conditions of Theorem~\ref{thm:regularfstability} contains at least $(1 - (30r - 35)\varepsilon)\frac{n}{2}$ components that are $\frac{r}{2}\cdot\cF$'s (see the proof of Theorem~\ref{thm:regularfstability}). In order to get the improved bound stated in the theorem, we will establish one more technical lemma.

Call a vertex \emph{$V_i$-bad} if it is part of a component of $\lk V_i$ that is not an $r$-regular $C_4$. Call a vertex \emph{bad} if it is $V_i$-bad for some $i$, and call a vertex \emph{good} otherwise.

\begin{lem} \label{lem:onebadvertex}
Let $\cH$ be a $3$-partite $3$-graph of maximum degree $r$ with vertex classes $V_1$, $V_2$, and $V_3$. Let $\lset{i, j, k} = \lset{1, 2, 3}$. If an $r$-regular $C_4$ component of $\lk V_i$ is such that all of its vertices are good except one $V_k$-bad vertex in $V_j$, then it shares vertices of $V_k$ with two $r$-regular $C_4$ components of $\lk V_j$ that each have two bad vertices (one $V_i$-bad, and one $V_k$-bad), and shares one vertex of $V_j$ with an $r$-regular $C_4$ component of $\lk V_k$ that has exactly one $V_i$-bad vertex in $V_j$. Furthermore, these four $r$-regular $C_4$ components do not share vertices with any $r$-regular $C_4$ component outside of these four.
\end{lem}

\begin{proof}
We know by Lemma~\ref{lem:extendtof} that such a $C_4$ component must be host to two disjoint edges of $\cH$, otherwise it would extend to an $\frac{r}{2}\cdot\cF$ and all of its links would be $r$-regular $C_4$'s. Thus, let $a_1 a_2 a_3$ and $b_1 b_2 b_3$ be two disjoint edges of $\cH$ with $a_i, b_i \in V_i$ and $a_2$, $b_2$, $a_3$, and $b_3$ being the vertices of an $r$-regular $C_4$ component of $\lk V_1$, all of whose vertices are part of $r$-regular $C_4$ components in the other links except for $b_3$. We consider two cases:

\noindent \emph{Case 1}. $a_1 a_2$ and $b_1 b_2$ belong to the same $r$-regular $C_4$ component of $\lk V_3$.

In this case, all edges incident to $a_1$ or $b_1$ are incident to $a_2$ or $b_2$, hence incident to $a_3$ or $b_3$, and vice versa. But this means that the $r$-regular $C_4$ component of $\lk V_2$ that $a_3$ participates in must have $\lset{a_1, b_1, a_3, b_3}$ as its vertex set, which contradicts the fact that $b_3$ is not in an $r$-regular $C_4$ component of $\lk V_2$. Therefore, this case is impossible.

\noindent \emph{Case 2}. $a_1 a_2$ and $b_1 b_2$ belong to two different $r$-regular $C_4$ components of $\lk V_3$.

In this case, let the vertices of the components be $a_1$, $c_1$, $a_2$, $c_2$, and $b_1$, $d_1$, $b_2$, $d_2$, respectively. Now consider $\lk V_2$. It has edges $a_1 a_3$ and $b_1 b_3$. Note that these edges are in separate components of $\lk V_2$, since $a_3$ participates in an $r$-regular $C_4$, while $b_3$ doesn't. Therefore, there are no edges $a_1 b_3$ or $b_1 a_3$ in $\lk V_2$, which implies that all edges parallel to $a_2 b_3$ in $\lk V_1$ extend to $c_1$, rather than $a_1$ (these being the only two options given by $\lk V_3$), and similarly all edges parallel to $b_2 a_3$ in $\lk V_1$ extend to $d_1$ (not $b_1$). These edges of $\cH$ correspond to edges $c_1 b_3$ and $d_1 a_3$, respectively, in $\lk V_2$. Now by assumption, $a_3$ is part of an $r$-regular $C_4$ component of $\lk V_2$, and given the edges we already have shown to exist, we know three of its vertices. Denote the remaining vertex by $d_3$ so that $\lset{a_1, d_1, a_3, d_3}$ is the vertex set of that component.

Since $a_3$ and $c_1$ are in distinct components of $\lk V_2$, we see that all edges of $\cH$ corresponding to $a_2 a_3$ extend to $a_1$. Similarly, all edges corresponding to $b_2 b_3$ extend to $b_1$, all the ones corresponding to $a_2 b_3$ extend to $c_1$, and $b_2 a_3$ to $d_1$. Now in $\lk V_2$ there is at least one edge $a_1 d_3$. Any such edge does not extend to $a_2$ as seen in $\lk V_1$, and hence must extend to $c_2$ by considering $\lk V_3$. Similarly, the edges parallel to $d_1 d_3$ in $\lk V_2$ must extend to $d_2$.

Since $b_1 b_3$ and $c_1 b_3$ are edges of $\lk V_2$ in the component of $b_3$, which is not an $r$-regular $C_4$, we have that $b_1$ and $c_1$ are both $V_2$-bad vertices. We claim that $c_2$ and $d_2$ are $V_1$-bad vertices. Suppose to the contrary that they were good. Then by the existence of edges $c_2 d_3$ and $d_2 d_3$ in $\lk V_1$, these are part of the same $r$-regular $C_4$ component of $\lk V_1$. Call its fourth vertex $c_3$. Now any edge parallel to $c_2 c_3$ in $\lk V_1$ extends to $c_1$, since it may only extend to $c_1$ or $a_1$ by $\lk V_3$, and can't extend to $a_1$ by $\lk V_2$. Similarly, any edge parallel to $d_2 c_3$ in $\lk V_1$ extends to $b_1$. We just showed that all edges on $c_3$ go to $c_1$ or $b_1$ in $\lk V_2$. What we showed earlier is that all edges on $b_3$ go to $c_1$ or $b_1$ in $\lk V_2$. These account for all edges on $c_3$ and $b_3$, putting $b_3$ in an $r$-regular $C_4$ component, which is a contradiction, because $b_3$ was assumed not to participate in one of those in $\lk V_2$. Therefore, the component of $\lk V_1$ including $c_2$ and $d_2$ is not an $r$-regular $C_4$, hence these are $V_1$-bad vertices.

Thus, we have found two $r$-regular $C_4$ components of $\lk V_3$ with two bad vertices each: $\lset{a_1, c_1, a_2, c_2}$ harbours an $r$-regular $C_4$ with bad vertices $c_1$ and $c_2$, while $\lset{b_1, d_1, b_2, d_2}$ harbours an $r$-regular $C_4$ with bad vertices $b_1$ and $d_2$. We also have an $r$-regular $C_4$ in $\lk V_2$ on $\lset{a_1, d_1, a_3, d_3}$ with a single $V_1$-bad vertex $d_3$. Since all of the good vertices of these four $r$-regular $C_4$ components are shared among themselves, this proves the lemma.
\end{proof}

\begin{proof}[Proof of Theorem~\ref{thm:regularfstability}]
Let $\cH$ be an $r$-regular $3$-partite $3$-graph with $n$ vertices in each class, and assume $\nu(\cH) \leq (1 + \varepsilon)\frac{n}{2}$. Let $V_1$, $V_2$, and $V_3$ be the vertex classes of $\cH$.

First, we modify $\cH$ by replacing each component of $\cH$ that has a perfect matching with $r$ parallel copies of the perfect matching. Note that this does not change $\nu(\cH)$ nor the number of vertices in each class, and keeps $\cH$ $r$-regular. This change also clearly does not create any new copies of $\frac{r}{2}\cdot\cF$, so if we prove that the modified hypergraph has some number of $\frac{r}{2}\cdot\cF$ components, these must have been present in $\cH$ to begin with. Thus, we may assume that every perfect matching component of $\cH$ is just $r$ parallel copies of an edge.

For each $i$, by applying Lemma~\ref{lem:c4links} with $A = V_i$, we have that $\lk V_i$ contains at least $(1 - (6r - 7)\varepsilon)\frac{n}{2}$ components that are $r$-regular $C_4$'s. Call an $r$-regular $C_4$ component of a link \emph{good} if it contains no bad vertices, and \emph{ruined} otherwise. We claim that at least one of the links has at least $(1 - \left(22r - \frac{77}{3}\right)\varepsilon)\frac{n}{2}$ good $r$-regular $C_4$ components.

Since each link has in each vertex class at least $(1 - (6r - 7)\varepsilon)n$ vertices belonging to $r$-regular $C_4$ components, each link contributes at most $(6r - 7)\varepsilon n$ bad vertices to any vertex class. If the bad vertices in each vertex class each ruin a different $r$-regular $C_4$ component of one link, then we may have as many as $(12r - 14)\varepsilon n$ ruined $r$-regular $C_4$ components in that link, leaving us with only $(1 - (30r - 35)\varepsilon)\frac{n}{2}$ good components. But then that link has many $r$-regular $C_4$ components with only one bad vertex, so by Lemma~\ref{lem:onebadvertex}, the other links must have many such components with at least two bad vertices, and so these links will have more good components.

To make this precise, we count the total number of bad vertices in all three links. As we have seen, each link contributes at most $(6r - 7)\varepsilon n$ bad vertices to each vertex class. Since there are two vertex classes per link and three links total, we have at most $6(6r - 7)\varepsilon n$ bad vertices in all. Now let $x_i$ count the number of $r$-regular $C_4$ components of $\lk V_i$ with exactly one bad vertex, and let $y_i$ count the number of $r$-regular $C_4$ components of $\lk V_i$ with at least two bad vertices. Let $x = x_1 + x_2 + x_3$ and let $y = y_1 + y_2 + y_3$. Note that any bad vertex contributes to at most one of $x_1$, $x_2$, $x_3$, $y_1$, $y_2$, and $y_3$, since in one of the two links containing that vertex, it is in an $r$-regular $C_4$ component. Therefore, we find that $x + 2y \leq 6(6r - 7)\varepsilon n$, as there must be at least $x + 2y$ bad vertices. Now by Lemma~\ref{lem:onebadvertex}, every $r$-regular $C_4$ component with only one bad vertex appears together with another $r$-regular $C_4$ component with only one bad vertex and two $r$-regular $C_4$ components with two bad vertices each, and these four form a unit that does not touch any other such unit (hence there is no overlap in our counting). This implies that there must be at least as many $r$-regular $C_4$ components with two bad vertices as there are ones with only one bad vertex, hence $y \geq x$.

Now let $V_i$ be the vertex class such that $x_i$ is the least among $x_1$, $x_2$, and $x_3$. We thus have $x_i \leq \frac{x}{3}$. And since $3x \leq x + 2y \leq 6(6r - 7)\varepsilon n$, we have $x_i \leq \frac{2}{3}(6r - 7)\varepsilon n$. Now $\lk V_i$ has at most $2(6r - 7)\varepsilon n$ bad vertices that were contributed from the other two links, which leaves at most $2(6r - 7)\varepsilon n - x_i$ bad vertices to ruin the $r$-regular $C_4$ components counted by $y_i$. Since these each use at least two of these vertices, we have $y_i \leq \frac{1}{2}(2(6r - 7)\varepsilon n - x_i)$. Combining our inequalities we find that $\lk V_i$ therefore has $x_i + y_i \leq (6r - 7)\varepsilon n + \frac{1}{2}x_i \leq \frac{4}{3}(6r - 7)\varepsilon n$ ruined $r$-regular $C_4$ components. The rest must be good, so we have at least $(1 - (6r - 7)\varepsilon)\frac{n}{2} - \frac{4}{3}(6r - 7)\varepsilon n = (1 - \left(22r - \frac{77}{3}\right)\varepsilon)\frac{n}{2}$ good $r$-regular $C_4$ components in $\lk V_i$.

If any good $r$-regular $C_4$ component hosts two disjoint edges of $\cH$, then by Lemma~\ref{lem:perfectmatchingcomp} it is part of a perfect matching component of $\cH$, which is a contradiction, since we replaced these by parallel copies of a matching (so their links do not contain any $r$-regular $C_4$ components). Therefore, all good $r$-regular $C_4$ components extend to copies of $\frac{r}{2}\cdot\cF$ by Lemma~\ref{lem:extendtof}, so we have found the desired number of those in $\cH$, completing the proof.
\end{proof}

\begin{proof}[Proof of Theorem~\ref{thm:abcfstability}]
This follows along very similar lines as the proof of Theorem~\ref{thm:regularfstability}.

Let $\cH$ be a $3$-partite $3$-graph with vertex classes $A$, $B$, and $C$, such that $\abs{A} = n$, and suppose that every vertex of $A$ has degree at least $r$, and that every vertex in $B \cup C$ has degree at most $r$. Assume that $\nu(\cH) \leq (1 + \varepsilon)\frac{n}{2}$.

First, we modify $\cH$ by removing edges from vertices of $A$ that have degree strictly larger than $r$ until every vertex of $A$ has degree exactly $r$. Note that this does not hurt any of our assumptions and cannot create copies of $\frac{r}{2}\cdot\cF$. After this modification, $\cH$ has maximum degree $r$.

Next, we again modify $\cH$ (as in the proof of Theorem~\ref{thm:regularfstability}) by replacing each component of $\cH$ that has a perfect matching with $r$ parallel copies of the perfect matching. Note that again, this change does not affect our assumptions, and also clearly does not create any new copies of $\frac{r}{2}\cdot\cF$. Thus, we may assume that every perfect matching component of $\cH$ is just $r$ parallel copies of an edge.

Now apply Lemma~\ref{lem:c4links} to $\cH$ to find that $\lk A$ contains at least $(1 - (6r - 7)\varepsilon)\frac{n}{2}$-many $r$-regular $C_4$ components. Now delete from $\cH$ all vertices of $B$ and $C$ that are not in one of the $r$-regular $C_4$ components. This leaves at least $n' = (1 - (6r - 7)\varepsilon)n$ vertices in each of these classes. Note that all vertices of $B$ and $C$ now have degree $r$.

Next, we follow along the lines of the proof of Lemma~\ref{lem:c4links} to find out about $r$-regular $C_4$ components of $\lk B$ and $\lk C$. There must be some $S \subseteq B$ such that $\eta(L(\lk S)) \leq \abs{S} - (\abs{B} - \nu(\cH))$, otherwise $\cH$ would have a matching larger than $\nu(\cH)$ by Theorem~\ref{thm:hallsforhypergraphs}. We have $\nu(\lk S) \geq \abs{S}$, so by Corollary~\ref{cor:connbound}, if $\lk S$ has $k$-many $r$-regular $C_4$ components, then
\[
	\eta(L(\lk S)) \geq \frac{(2r - 3)\abs{S} + r\abs{S} - k}{6r - 7}.
\]
Combining this with our upper bound, we find
\begin{align*}
	k &\geq (6r - 7)(\abs{B} - \nu(\cH)) - (3r - 4)\abs{S}\\
	&\geq (6r - 7)\left(n' - (1 + \varepsilon)\frac{n}{2}\right) - (3r - 4)n'\\
	&= (1 - (36r^2 - 72r + 35)\varepsilon)\frac{n}{2}.
\end{align*}
Since $\lk B$ has maximum degree $r$, these components of $\lk S$ are all components of $\lk B$, hence we have found at least $(1 - (36r^2 - 72r + 35)\varepsilon)\frac{n}{2}$-many $r$-regular $C_4$ components in $\lk B$. The same holds for $\lk C$.

Call an $r$-regular $C_4$ component of a link \emph{good} if it contains no bad vertices, and \emph{ruined} otherwise. We claim that $\lk A$ has at least $(1 - (72r^2 - 150r + 77)\varepsilon)\frac{n}{2}$ good $r$-regular $C_4$ components.

Note that there are no $A$-bad vertices, since we deleted them all before considering $\lk B$ and $\lk C$. This means that all ruined $r$-regular $C_4$ components of $\lk A$ have at least two bad vertices, since if they only had one, Lemma~\ref{lem:onebadvertex} would imply the existence of an $A$-bad vertex (in fact, three of them). There are at most $n' - (1 - (36r^2 - 72r + 35)\varepsilon)n = (36r^2 - 78r + 42)\varepsilon n$-many $B$-bad vertices in $C$, and also no more than that many $C$-bad vertices in $B$. Since the ruined $r$-regular $C_4$ components of $\lk A$ each have two bad vertices, this means that there are in fact at most $(36r^2 - 78r + 42)\varepsilon n$ ruined $r$-regular $C_4$ components in $\lk A$. Therefore, since the rest are good, there are indeed at least $(1 - (72r^2 - 150r + 77)\varepsilon)\frac{n}{2}$ good $r$-regular $C_4$ components in $\lk A$.

If any good $r$-regular $C_4$ component hosts two disjoint edges of $\cH$, then by Lemma~\ref{lem:perfectmatchingcomp} it is part of a perfect matching component of $\cH$, which is a contradiction, since we replaced these by parallel copies of a matching (so their links do not contain any $r$-regular $C_4$ components). Therefore, all good $r$-regular $C_4$ components extend to copies of $\frac{r}{2}\cdot\cF$ by Lemma~\ref{lem:extendtof}, so we have found the desired number of those in $\cH$, completing the proof.
\end{proof}

\section{$\frac{r}{2}\cdot\cF$-Free $3$-Graphs}\label{last}

Theorems~\ref{thm:regularfstability} and ~\ref{thm:abcfstability} have the following easy corollaries, respectively:

\begin{cor} \label{cor:ffreenubound}
Let $\cH$ be an $r$-regular $3$-partite $3$-graph with $n$ vertices in each vertex class. If $\cH$ does not contain a copy of $\frac{r}{2}\cdot\cF$, then
\[
	\nu(\cH) \geq \left(1 + \frac{1}{22r - \frac{77}{3}}\right)\frac{n}{2}.
\]
\end{cor}

\begin{cor} \label{cor:abcffreenubound}
Let $\cH$ be a $3$-partite $3$-graph with vertex classes $A$, $B$, and $C$, such that $\abs{A} = n$. Suppose that every vertex of $A$ has degree at least $r$, and that every vertex in $B \cup C$ has degree at most $r$. If $\cH$ contains no subgraph isomorphic to $\frac{r}{2}\cdot\cF$, then
\[
	\nu(\cH) \geq \left(1 + \frac{1}{72r^2 - 150r + 77}\right)\frac{n}{2}.
\]
\end{cor}

This answers the question of Aharoni, Kotlar, and Ziv~\cite{aharonikotlarziv} mentioned in the introduction, since for $r \geq 3$, any simple $3$-partite $3$-graph is $\frac{r}{2}\cdot\cF$-free. 

It would be interesting to determine the correct function $\alpha(r)$ for which $\nu(\cH) \geq (1 + \alpha(r))\frac{n}{2}$ for every $\cH$ satisfying the conditions of Corollary~\ref{cor:ffreenubound}. The following constructions give upper bounds on $\alpha(r)$. 

\begin{thm} \label{thm:ffreenuexamples}
For every even $r \geq 2$ there exists an $r$-regular $3$-partite $3$-graph $\cH$ with $n$ vertices per vertex class, not containing a copy of $\frac{r}{2}\cdot\cF$, such that
\[
	\nu(\cH) \leq \left(1 + \frac{1}{r + 1}\right)\frac{n}{2}.
\]
For every odd $r \geq 3$ there exists an $r$-regular $3$-partite $3$-graph $\cH$ with $n$ vertices per vertex class (obviously not containing a copy of $\frac{r}{2}\cdot\cF$) such that
\[
	\nu(\cH) \leq \left(1 + \frac{1}{r}\right)\frac{n}{2}.
\]
\end{thm}

\begin{proof}
First suppose $r \geq 2$ is even. Let $\frac{r}{2}\cdot\cF^-$ denote the $3$-partite $3$-graph obtained by removing a single edge from $\frac{r}{2}\cdot\cF$. Note that it has three vertices of degree $r - 1$ and three vertices of degree $r$. Take $\frac{r}{2}$ disjoint copies of $\frac{r}{2}\cdot\cF^-$ together with three vertices $a$, $b$, and $c$, one in each class. For each copy $F$ of $\frac{r}{2}\cdot\cF^-$, add three edges, each using two of $a$, $b$, and $c$ and one of the three degree-$(r - 1)$ vertices of $F$. Each group of three edges contributes $2$ to the degree of $a$, $b$, and $c$, and $1$ to the degree of the degree-$(r - 1)$ vertices, hence after all $\frac{r}{2}$ such groups are added, the resulting $3$-graph is $r$-regular and clearly $\frac{r}{2}\cdot\cF$-free. It has $n = r + 1$ vertices per vertex class, and its largest matching is of size at most $\frac{r}{2} + 1$, since in any matching we can pick at most one edge from each copy of $\frac{r}{2}\cdot\cF^-$, and all of the edges we added intersect in one of $a$, $b$, or $c$. This gives the desired bound for even $r$.

If $r \geq 3$ is odd, we can use a very similar construction as above. Instead of $\frac{r}{2}\cdot\cF^-$, which does not exist for odd $r$, let $\frac{r - 1}{2}\cdot\cF^+$ denote the $3$-partite $3$-graph obtained from $\frac{r - 1}{2}\cdot\cF$ by adding an extra copy of one of its edges. Note that it has three vertices of degree $r - 1$ and three vertices of degree $r$. Taking $\frac{r - 1}{2}$ disjoint copies of $\frac{r - 1}{2}\cdot\cF^+$ together with three vertices $a$, $b$, and $c$, one in each class, we add edges containing two of these vertices and one degree-$(r - 1)$ vertex of an $\frac{r - 1}{2}\cdot\cF^+$ as in the previous construction. We also add the edge $abc$. The resulting $3$-graph is $r$-regular and clearly $\frac{r}{2}\cdot\cF$-free (since this $3$-graph does not exist for odd $r$). It has $n = r$ vertices per vertex class, and its largest matching is of size at most $\frac{r - 1}{2} + 1$, since we can pick at most one edge from each copy of $\frac{r - 1}{2}\cdot\cF^+$, and all of the edges we added intersect in one of the three extra vertices $a$, $b$, and $c$. This gives the desired bound for odd $r$.
\end{proof}

All of these examples have high edge multiplicity, and as mentioned in the introduction, one may expect substantially better lower bounds on the matching number for \emph{simple} hypergraphs. We close with the following conjectures about this more restrictive case.

\begin{conj}[Aharoni, Kotlar and Ziv~\cite{aharonikotlarziv}]
Let $\cH$ be an $r$-regular simple $3$-partite $3$-graph with $n$ vertices in each class. Then $\nu(\cH) \geq \frac{r - 1}{r}n$.
\end{conj}

\begin{conj}[Aharoni, Berger, Kotlar and Ziv~\cite{ABKZ}]
Let $\cH$ be a simple $3$-partite $3$-graph with vertex classes $A$, $B$ and $C$. Suppose each vertex in $A$ has degree at least $r$, and each vertex in $B \cup C$ has degree at most $r$. Then $\nu(\cH) \geq \frac{r - 1}{r}\abs{A}$.
\end{conj}

These conjectures for $r = n$ generalize a notorious old open problem of  Ryser-Brualdi-Stein on Latin transversals, so in their full generality they are likely to be very difficult.

\end{document}